\documentclass[11pt]{amsart}%
\usepackage{graphicx}
\usepackage{enumerate}
\usepackage{amsmath}
\usepackage{amsfonts}
\usepackage{amssymb}
\usepackage{url}
\usepackage[margin=1in]{geometry}%
\providecommand{\U}[1]{\protect\rule{.1in}{.1in}}
\theoremstyle{plain}
\newtheorem{theorem}{Theorem}[section]
\newtheorem{corollary}[theorem]{Corollary}

\newtheorem{example}[theorem]{Example}

\newtheorem{definition}[theorem]{Definition}
\newtheorem{problem}[theorem]{Problem}

\numberwithin{equation}{section}
\begin{document}
\title{Self-concordance is NP-hard}
\author{Lek-Heng~Lim}
\address{Computational and Applied Mathematics Initiative, Department of Statistics,
University of Chicago, Chicago, IL 60637-1514.}
\email{lekheng@galton.uchicago.edu}

\begin{abstract}
We give an elementary proof of a somewhat curious result, namely, that
deciding whether a convex function is self-concordant is in general an
intractable problem.

\end{abstract}
\maketitle

\section{Introduction}

Nesterov and Nemirovskii \cite{NN} famously showed that the optimal solution
of a conic programming problem can be computed to $\varepsilon$-accuracy in
polynomial time if the cone has a self-concordant barrier function whose
gradient and Hessian are both computable in polynomial time. Their work
established self-concordance as a singularly important notion in modern
optimization theory.

We show in this article that deciding whether a convex function is
self-concordant at a point is nonetheless an NP-hard problem. In fact we will
prove that deciding the self-concordance of a convex function defined locally
by a cubic polynomial (which cannot be convex on all of $\mathbb{R}^{n}$),
arguably the simplest non-trivial instance, is already an NP-hard problem.
This is not so surprising given that deciding the convexity of a quartic
polynomial, arguably the simplest non-trivial instance of deciding convexity,
is also NP-hard --- a much harder recent result \cite{AOPT}. In addition to
the NP-hardness of self-concordance, we will see that there is no fully
polynomial time approximation scheme for the optimal self-concordant parameter
and that deciding second-order self-concordance \cite{J} of a quartic
polynomial is also an NP-hard problem.

These hardness results are intended only to add to our understanding of
self-concordance. They do not in anyway detract from the usefulness of the
notion since in practice self-concordant barriers are \textit{constructed} at
the outset to have the requisite property \cite[Chapter~5]{NN}. It is unlikely
that one would ever need to generate random functions and then test them for self-concordance.

We deduce the NP-hardness of self-concordance using a result of Nesterov
himself, namely, minimizing a cubic form over a sphere is in general an
NP-hard problem. Nesterov's result, which appears in an unpublished manuscript
\cite{N}, contains some minor errors that have unfortunately been widely
reproduced. We take the opportunity here to correct them and also to prove a
variant of Nesterov's result directly from Motzkin--Straus Theorem \cite{MS}.

\section{Self-concordance in terms of tensors\label{sec:sc}}

Let $\Omega\subseteq\mathbb{R}^{n}$ be open and $f:\Omega\rightarrow
\mathbb{R}$ be in $C^{d}(\Omega)$, i.e., has continuous partials up to at
least order $d$. Recall that the $d$th order derivative of $f$ at $x\in\Omega
$, denoted $\nabla^{d}f(x)$, is a \textit{tensor} of order $d$ \cite{Lang}. To
be more precise, this simply means that $\nabla^{d}f(x)$ is a multilinear
functional on $T_{x}(\Omega)$, the tangent space of $\Omega$ at $x$, that is,%
\[
\nabla^{d}f(x):\underset{d\text{ copies}}{\underbrace{T_{x}(\Omega)\times
\dots\times T_{x}(\Omega)}}\rightarrow\mathbb{R},
\]
where
\begin{multline*}
\nabla^{d}f(x)(h_{1},\dots,\alpha h_{i}+\beta h_{i}^{\prime},\dots
,h_{n})=\alpha\nabla^{d}f(x)(h_{1},\dots,h_{i},\dots,h_{n})\\
+\beta\nabla^{d}f(x)(h_{1},\dots,h_{i}^{\prime},\dots,h_{n})\qquad\text{for
}i=1,\dots,n.
\end{multline*}
With respect to the standard basis $\frac{\partial}{\partial x_{1}}%
,\dots,\frac{\partial}{\partial x_{n}}$ of $T_{x}(\Omega)$, we may identify
$T_{x}(\Omega)\cong\mathbb{R}^{n}$ and $\nabla^{d}f(x)$ may be regarded as a
`$d$-dimensional matrix' ($d$\textit{-hypermatrix} for short),%
\[
\nabla^{d}f(x)=[a_{i_{1}\cdots i_{d}}]_{i_{1},\dots,i_{d}=1}^{n}\in
\mathbb{R}^{n\times\dots\times n}.
\]
Indeed, we must have%
\[
a_{i_{1}\cdots i_{d}}=\frac{\partial^{d}f(x)}{\partial x_{i_{1}}\cdots\partial
x_{i_{d}}},
\]
and since $f\in C^{d}(\Omega)$, we get that $a_{i_{1}\cdots i_{d}%
}=a_{i_{\sigma(1)}\cdots i_{\sigma(d)}}$, i.e., $\nabla^{d}f(x)$ is a
\textit{symmetric} $d$-hypermatrix. Every symmetric $d$-hypermatrix
$A=[a_{i_{1}\cdots i_{d}}]_{i_{1},\dots,i_{d}=1}^{n}\in\mathbb{R}%
^{n\times\dots\times n}$ defines a homogeneous polynomial of degree $d$,
denoted%
\[
A(h,\dots,h):=\sum_{i_{1},\dots,i_{d}=1}^{n}a_{i_{1}\cdots i_{d}}h_{i_{1}%
}\cdots h_{i_{d}}\in\mathbb{R}[h_{1},\dots,h_{n}]_{d}.
\]
Readers should be able to infer from the above discussion that $d$%
-hypermatrices are coordinate representations of $d$-tensors, just as matrices
are coordinate representations of linear operators and bilinear forms (both
are $2$-tensors). We refer the reader to \cite{L} for more information.

The usual definition of self-concordance requires that $f\in C^{3}(\Omega)$
and in which case it is given by a condition involving the matrix $\nabla
^{2}f(x)\in\mathbb{R}^{n\times n}$ and the $3$-hypermatrix $\nabla^{3}%
f(x)\in\mathbb{R}^{n\times n\times n}$.

\begin{definition}
[Nesterov--Nemirovskii]\label{def:sc}Let $\Omega\subseteq\mathbb{R}^{n}$ is a
convex open set. Then $f:\Omega\rightarrow\mathbb{R}$ is said to be
\textbf{self-concordant} with parameter $\sigma>0$ at $x\in\Omega$ if
\begin{equation}
\nabla^{2}f(x)(h,h)\geq0\label{eq:2nd}%
\end{equation}
and%
\begin{equation}
\lbrack\nabla^{3}f(x)(h,h,h)]^{2}\leq4\sigma\lbrack\nabla^{2}f(x)(h,h)]^{3}%
\label{eq:concord}%
\end{equation}
for all $h\in\mathbb{R}^{n}$; $f$ is self-concordant on $\Omega$ if
\eqref{eq:2nd} and \eqref{eq:concord} hold for all $x\in\Omega$. The set of
self-concordant functions on $\Omega$ with parameter $\sigma$ is denoted by
$S_{\sigma}(\Omega)$.
\end{definition}

By \eqref{eq:2nd}, a function self-concordant on $\Omega$ is necessarily
convex on $\Omega$. A minor deviation from \cite{NN} is that $\sigma$ above is
really the reciprocal of the self-concordance parameter as defined in
\cite[Definition~2.1.1]{NN}. Our hardness results would be independent of the
choice of $\sigma$. Note that%
\[
\nabla^{2}f(x)(h,h)=\sum_{i,j=1}^{n}\frac{\partial^{2}f(x)}{\partial
x_{i}\partial x_{j}}h_{i}h_{j},\quad\nabla^{3}f(x)(h,h,h)=\sum_{i,j,k=1}%
^{n}\frac{\partial^{3}f(x)}{\partial x_{i}\partial x_{j}\partial x_{k}}%
h_{i}h_{j}h_{k}.
\]
So for a fixed $x\in\Omega$, $\nabla^{2}f(x)(h,h)$ is a quadratic form in $h$
and $\nabla^{3}f(x)(h,h,h)$ is a cubic form in $h$. It is well-known that
deciding \eqref{eq:2nd} at any fixed $x$ is a polynomial-time problem (but not
so for deciding it over all $x\in\Omega$, see \cite{AOPT}). Hence given a
$\sigma>0$, deciding self-concordance at $x$ essentially boils down to
\eqref{eq:concord}: Is the square of a given cubic form globally bounded above
by the cube of a given quadratic form? We shall see in the next sections that
this decision problem is NP-hard.

While we shall think of $\nabla^{2}f(x)$ as a matrix and $\nabla^{3}f(x)$ as a
hypermatrix thoughout this article, we nonetheless wish to highlight that
condition \eqref{eq:concord} is really a condition on $\nabla^{2}f(x)$
regarded as a $2$-tensor and $\nabla^{3}f(x)$ regarded as a $3$-tensor; that
is, \eqref{eq:concord} is independent of the choice of coordinates, a property
that follows from the \textit{affine invariance} of self-concordance
\cite[Proposition~2.1.1]{NN}. Self-concordance on $\Omega$ is then a global
condition about the tensor fields $\nabla^{3}f$ and $\nabla^{2}f$.

\section{Maximizing a cubic form over a sphere}

We will include a proof that the clique and stability numbers of a graph with
$n$ vertices and $m$ edges may be expressed as the maximal values of cubic
forms (in $n+m$ variables) over the unit sphere $\mathbb{S}^{n+m-1}$. This, or
at least the stability number version, is known but the reference
\cite[Theorem~4]{N} usually cited contains some slight errors that have been
reproduced several times elsewhere\footnote{For example, \cite[Theorem~3.4]%
{dK}. To see that the expression is incorrect, take a graph with three
vertices and one edge, the left-hand side gives $1/\sqrt{2}$ and the
right-hand side gives $1$.}. We take the opportunity to provide a corrected
version below. Our proof follows Motzkin--Straus Theorem \cite{MS} and the
similar result of Nesterov \cite[Theorem~4]{N} for stability number.

Let $G=(V,E)$ be an undirected graph with $n$ vertices and $m$ edges. We shall
require that $E\neq\varnothing$ throughout, so $n\geq2$ and $m\geq1$. Recall
that $S\subseteq V$ is a \textit{clique} in $G$ if $\{i,j\}\in E$ for all
$i,j\in S$ and $S\subseteq V$ is \textit{stable} in $G$ if $\{i,j\}\notin E$
for all $i,j\in S$. The \textit{clique number} and \textit{stability number}
of $G$ are respectively:
\[
\omega(G)=\max\{\lvert S\rvert:S\subseteq V\text{ is clique}\},\quad
\alpha(G)=\max\{\lvert S\rvert:S\subseteq V\text{ is stable}\}.
\]
Motzkin and Straus \cite{MS} showed that $\omega(G)$ may be expressed in terms
of the maximal value of a simple quadratic polynomial over the unit simplex.
Although not in \cite{MS}, it is straightforward to see that essentially the
same proof also yields a similar expression for $\alpha(G)$.

\begin{theorem}
[Motzkin--Straus]\label{thm:ms}Let $\Delta^{n}=\{x\in\mathbb{R}^{n}%
:x_{1}+\dots+x_{n}=1,\;x_{i}\geq0\}$ denote the unit simplex in $\mathbb{R}%
^{n}$. Then the clique number $\omega(G)$ and stability number $\alpha(G)$ may
be determined via quadratic optimization over simplices:%
\begin{equation}
1-\frac{1}{\omega(G)}=2\max_{x\in\Delta^{n}}\sum\nolimits_{\{i,j\}\in E}%
x_{i}x_{j},\quad1-\frac{1}{\alpha(G)}=2\max_{x\in\Delta^{n}}\sum
\nolimits_{\{i,j\}\notin E}x_{i}x_{j}. \label{eq4b}%
\end{equation}

\end{theorem}

Since deciding if a clique of a given size exists is an NP-complete problem
\cite{Ka}, an immediate consequence is that clique number is NP-hard and by
Motzkin--Straus Theorem, so is quadratic maximization over a simplex.

In an unpublished manuscript \cite[Theorem~4]{N}, Nesterov used
Motzkin--Straus Theorem to obtain an alternate expression \eqref{eq7} for
stability number involving the maximal value of a cubic form over a sphere. In
the following we derive a similar expression \eqref{eq6} for the clique
number, which yields slightly simpler expressions for our discussions in
Sections~\ref{sec:NP} and \ref{sec:2sc}, and may perhaps be of independent interest.

\begin{theorem}
[Nesterov]\label{thm:clique}Let $G=(V,E)$ be an undirected graph with $n$
vertices and $m$ edges. Let $\mathbb{S}^{d-1}=\{x\in\mathbb{R}^{d}:\lVert
x\rVert=1\}$ denote the unit $\ell^{2}$-sphere in $\mathbb{R}^{d}$. The clique
number $\omega(G)$ and stability number $\alpha(G)$ may be determined via
cubic optimization over spheres:%
\begin{align}
1-\frac{1}{\omega(G)}  &  =\frac{27}{2}\max_{(u,v)\in\mathbb{S}^{n+m-1}%
}\left[  \sum\nolimits_{\{i,j\}\in E}u_{i}u_{j}w_{ij}\right]  ^{2}%
,\label{eq6}\\
1-\frac{1}{\alpha(G)}  &  =\frac{27}{2}\max_{(u,v)\in\mathbb{S}^{n+m-1}%
}\left[  \sum\nolimits_{\{i,j\}\notin E}u_{i}u_{j}w_{ij}\right]  ^{2}.
\label{eq7}%
\end{align}

\end{theorem}

\begin{proof}
This follows from Motzkin--Straus Theorem and the equalities%
\begin{align}
\max_{x\in\Delta^{n}}\sum\nolimits_{\{i,j\}\in E}x_{i}x_{j}  &  =\max
_{u\in\mathbb{S}^{n-1}}\sum\nolimits_{\{i,j\}\in E}u_{i}^{2}u_{j}%
^{2}\label{eq4i}\\
&  =\max_{u\in\mathbb{S}^{n-1},w\in\mathbb{S}^{m-1}}\left[  \sum
\nolimits_{\{i,j\}\in E}u_{i}u_{j}w_{ij}\right]  ^{2}\label{eq4ii}\\
&  =\frac{27}{4}\max_{(u,v)\in\mathbb{S}^{n+m-1}}\left[  \sum
\nolimits_{\{i,j\}\in E}u_{i}u_{j}w_{ij}\right]  ^{2}. \label{eq4iii}%
\end{align}
\eqref{eq4i} comes from substituting $x_{i}=u_{i}^{2}$, $i=1,\dots,n$.
Cauchy-Schwartz yields
\[
\sum\nolimits_{\{i,j\}\in E}u_{i}u_{j}w_{ij}\leq\left(  \sum
\nolimits_{\{i,j\}\in E}u_{i}^{2}u_{j}^{2}\right)  ^{1/2}\left(
\sum\nolimits_{\{i,j\}\in E}w_{ij}^{2}\right)  ^{1/2}%
\]
and so%
\begin{align}
\max_{\lVert u\rVert=\lVert w\rVert=1}\sum\nolimits_{\{i,j\}\in E}u_{i}%
u_{j}w_{ij}  &  \leq\max_{\lVert u\rVert=1}\left(  \sum\nolimits_{\{i,j\}\in
E}u_{i}^{2}u_{j}^{2}\right)  ^{1/2}\max_{\lVert w\rVert=1}\left(
\sum\nolimits_{\{i,j\}\in E}w_{ij}^{2}\right)  ^{1/2}\nonumber\\
&  =\max_{\lVert u\rVert=1}\left(  \sum\nolimits_{\{i,j\}\in E}u_{i}^{2}%
u_{j}^{2}\right)  ^{1/2}=:\alpha. \label{eq3}%
\end{align}
Let the maximal value $\alpha$ be attained at $\overline{u}\in\mathbb{S}%
^{n-1}$. We set $\overline{w}_{ij}=\overline{u}_{i}\overline{u}_{j}/\alpha$
for all $\{i,j\}\in E$ (note that $\alpha>0$ if $E\neq\varnothing$). Observe
that%
\[
\sum\nolimits_{\{i,j\}\in E}\overline{u}_{i}\overline{u}_{j}\overline{w}%
_{ij}=\frac{1}{\alpha}\sum\nolimits_{\{i,j\}\in E}\overline{u}_{i}%
^{2}\overline{u}_{j}^{2}=\alpha,
\]
and $\overline{w}\in\mathbb{S}^{m-1}$ since%
\[
\sum\nolimits_{\{i,j\}\in E}\overline{w}_{ij}^{2}=\frac{1}{\alpha^{2}}%
\sum\nolimits_{\{i,j\}\in E}\overline{u}_{i}^{2}\overline{u}_{j}^{2}=1.
\]
Hence equality is attained in \eqref{eq3} and we have \eqref{eq4ii}. We deduce
\eqref{eq4iii} from%
\begin{align*}
&  \max_{\lVert(u,w)\rVert=1}\sum\nolimits_{\{i,j\}\in E}u_{i}u_{j}w_{ij} =
\max_{\lVert u\rVert^{2}+\lVert w\rVert^{2}=1}\sum\nolimits_{\{i,j\}\in
E}u_{i}u_{j}w_{ij}\\
=  &  \sup_{\beta\in(0,1)}\left[  \max_{\lVert u\rVert^{2}=\beta,\;\lVert
w\rVert^{2}=1-\beta}\sum\nolimits_{\{i,j\}\in E}u_{i}u_{j}w_{ij}\right] \\
=  &  \sup_{\beta\in(0,1)}\left[  \max_{\lVert u/\sqrt{\beta}\rVert
^{2}=1,\;\lVert w/\sqrt{1-\beta}\rVert^{2}=1}\sum\nolimits_{\{i,j\}\in E}%
\frac{u_{i}}{\sqrt{\beta}}\frac{u_{j}}{\sqrt{\beta}}\frac{w_{ij}}%
{\sqrt{1-\beta}}\times\beta\sqrt{1-\beta}\right] \\
=  &  \left[  \max_{\lVert u\rVert^{2}=1,\;\lVert w\rVert^{2}=1}%
\sum\nolimits_{\{i,j\}\in E}u_{i}u_{j}w_{ij}\right]  \times\sup_{\beta
\in(0,1)}\beta\sqrt{1-\beta}\\
=  &  \frac{2}{3\sqrt{3}}\max_{\lVert u\rVert=\lVert w\rVert=1}\sum
\nolimits_{\{i,j\}\in E}u_{i}u_{j}w_{ij}.
\end{align*}
Note that the maximal value of $\sum\nolimits_{\{i,j\}\in E}u_{i}u_{j}w_{ij}$,
whether over $\mathbb{S}^{n-1}\times\mathbb{S}^{m-1}$ or over $\mathbb{S}%
^{n+m-1}$, can always be attained with $u\geq0$ and $w\geq0$, thereby allowing
one to take square in \eqref{eq4ii} and \eqref{eq4iii}. The same proof works
word-for-word for stability number with the replacement of index of summation
`$\{i,j\}\in E$' by `$\{i,j\}\notin E$'.
\end{proof}

\section{Complexity theory for casual users\label{sec:ct}}

In this article, we use complexity classes defined in the standard classical
framework: Time complexity measured in bits with the Cook-Karp-Levin notions
of reducibility \cite{Co,Ka,Le} on a Turing machine \cite{Tu}. This is also
the most common framework for discussing complexity issues in optimization
\cite{V}. We briefly recall the intuitive ideas behind the complexity classes
used in this article for our readers. This is of course not meant to be
anywhere near a rigorous treatment; for that, see \cite{Papa, Sip}.

A decision problem, i.e., answer is \textsc{yes} or \textsc{no}, is in NP if
one can verify an \textsc{yes} answer in polynomial time; it is said to be in
co-NP if one can verify a \textsc{no} answer in polynomial time. A decision
problem in NP is said to be NP-complete if one can reduce any other problem in
NP to it. Likewise, a decision problem in co-NP is said to be co-NP-complete
if one can reduce any other problem in co-NP to it. NP-complete and
co-NP-complete problems are believed to be intractable.

\begin{example}
[Subset sum problem \cite{knapsack}]Let $S$ be a finite set of integers. The
problem ``Does $S$ have a non-empty subset with a zero sum?'' is NP-complete
--- one can check a purported \textsc{yes} answer, i.e., a nonempty subset of
integers, is indeed a \textsc{yes} answer, i.e., has a zero sum. The problem
``Does every non-empty subset of $S$ have a nonzero sum?'' is co-NP-complete
--- one can check a purported \textsc{no} answer, i.e., a nonempty subset of
integers, is indeed a \textsc{no} answer, i.e., has a nonzero sum. Note that
the two problems are logical complements of each other. This is in fact
another way to define the co-NP class, namely, it comprises problems that are
logical complements of NP problems.
\end{example}

A problem is said to be \textit{NP-hard} \cite{Kn1,Kn2} if one can reduce any
NP-complete decision problem to it. A problem is said to be
\textit{co-NP-hard} if one can reduce any co-NP-complete decision problem to
it. In other words, if one can solve an NP-hard problem, then one can solve
any NP (including NP-complete) problems; if one can solve a co-NP-hard
problem, then one can solve any co-NP (including co-NP-complete) problems. So
NP-hard problems are believe to be even harder than NP-complete problem and
co-NP-hard problems are believed to be even more intractable than
co-NP-complete problems. An NP-hard problem need not be in NP, i.e., one need
not be able to check an \textsc{yes} answer in polynomial time; an NP-hard
problem that is in NP is by definition NP-complete. Likewise a co-NP-hard
problem need not be in co-NP, i.e., one need not be able to check a
\textsc{no} answer in polynomial time; a co-NP-hard problem that is in co-NP
is by definition co-NP-complete.

We have used the term `reduce' without stating what it meant. There are in
fact two different notions: The \textit{Cook reduction} (also known as
polynomial-time Turing reduction) and the \textit{Karp reduction} (also known
as polynomial-time many-one reduction). For our purpose, all we need to know
is that Cook reduction is believed to be a stronger notion of reducibility
than Karp reduction: Under Cook reduction, every decision problem can be
reduced to its complement. A consequence is that with respect to Cook
reducibility, there is no distinction between NP and co-NP, NP-complete and
co-NP-complete, or NP-hard and co-NP-hard. Under Karp reduction, it is
believed (although unknown) that these classes are all distinct.

Given an NP-hard maximization problem, e.g., the ones in Theorems~\ref{thm:ms}
and \ref{thm:clique}, a \textit{polynomial time approximation scheme}
(\textsc{ptas}) is an algorithm that, for any fixed $\varepsilon>0$, would
produce an approximate solution within a factor of $1-\varepsilon$ of the
maximal value, and has running time polynomial in input size and $\varepsilon$.

\section{Complexity of deciding self-concordance\label{sec:NP}}

The recent resolution of Shor's conjecture by Ahmadi, Olshevsky, Parrilo, and
Tsitsiklis \cite{AOPT} shows that deciding the convexity of a quartic
polynomial globally over $\mathbb{R}^{n}$ is NP-hard. So the self-concordance
of a function that is not a priori known to be convex is NP-hard in a trivial
way since deciding whether \eqref{eq:2nd} holds for all $x\in\Omega$ in
Definition~\ref{def:sc} is already an NP-hard problem. Our complexity result
requires more stringent conditions: (i) Our functions are assumed to be convex
in $\Omega$ and so \eqref{eq:2nd} is always satisfied and self-concordance
reduces to checking \eqref{eq:concord}; (ii) We show that \eqref{eq:concord}
is already NP-hard to check at a single point $x\in\Omega$.

Throughout the following we shall require the \textit{inputs} to our problems
to take rational or finite-extensions\footnote{In this article we only
encounter simple quadratic and cubic extensions, cf.\ \eqref{eq:quad} and
\eqref{eq:cube}. Note that for positive $q\in\mathbb{Q}$, elements of
$\mathbb{Q}(\sqrt{q})$ and $\mathbb{Q}(\sqrt[3]{q})$ may be written as
$a+b\sqrt{q}$ and $a+b\sqrt[3]{q}+c(\sqrt[3]{q})^{2}$ respectively. Therefore
they may be represented by pairs and triples of rational numbers.} of rational
values, e.g., $A\in\mathbb{Q}^{n\times n\times n},q\in\mathbb{Q}$, to ensure a
finite bit-length input. Note however that an NP-hard problem may not be in
the class NP; so an NP-hard decision problem can be posed over the reals,
e.g., `Is there an $h\in\mathbb{R}^{n}$ such that $[A(h,h,h)]^{2}\leq
q[h^{\top}h]^{3}$ holds?' without causing any issue since it is not required
to have a polynomial-time checkable certificate.

We will now formulate a decision problem that will lead us to the requisite
NP-hardness of self-concordance. Let $G=(V,E)$ be an undirected graph with $n$
vertices and $m$ edges where $n\geq2$ and $m\geq1$. Let $E=\{\{i_{k}%
,j_{k}\}:k=1,\dots,m\}$. Define $A_{G}=[a_{ijk}]_{i,j,k=1}^{n+m}\in
\mathbb{Q}^{(n+m)\times(n+m)\times(n+m)}$ by%
\[
a_{ijk}=%
\begin{cases}
1 & \{i_{k},j_{k}\}\in E\text{,}\\
0 & \text{otherwise.}%
\end{cases}
\]
Note that $A_{G}$ is a symmetric hypermatrix, i.e.,%
\[
a_{ijk}=a_{ikj}=a_{jik}=a_{jki}=a_{kij}=a_{kji}%
\]
for all $i,j,k\in\{1,\dots,n+m\}$. Let us denote the coordinates of
$h\in\mathbb{R}^{n+m}$ as
\[
h=(u_{1},\dots,u_{n},w_{i_{1}j_{1}},\dots
,w_{i_{m}j_{m}}).
\]
In which case,
\[
A_{G}(h,h,h)=\sum\nolimits_{k=1}^{m}u_{i_{k}}u_{j_{k}}w_{i_{k}j_{k}}%
=\sum\nolimits_{\{i,j\}\in E}u_{i}u_{j}w_{ij},
\]
and so by Theorem~\ref{thm:clique},%
\begin{equation}
\max_{h\neq0}\left[  \frac{A_{G}(h,h,h)}{\lVert h\rVert^{3}}\right]  ^{2}%
=\max_{\lVert h\rVert=1}[A_{G}(h,h,h)]^{2}=\frac{2}{27}\left(  1-\frac
{1}{\omega(G)}\right)  .\label{eq8}%
\end{equation}
The \textsc{clique} problem asks if for a given graph $G$ and a given
$k\in\mathbb{N}$, whether $G$ has a clique of size $k$? \textsc{clique} is
well-known to be NP-complete \cite{Ka}. In other words, deciding if
$\omega(G)\geq k$, or equivalently, $\omega(G)>k-1$, is an NP-hard problem;
and by \eqref{eq8}, so is deciding if%
\begin{equation}
\max_{h\neq0}\left[  \frac{A_{G}(h,h,h)}{\lVert h\rVert^{3}}\right]
^{2}>\frac{2}{27}\left(  1-\frac{1}{k-1}\right)  .\label{eq9}%
\end{equation}
Strictly speaking, we have restricted ourselves to the subclass of undirected
graphs with at least two vertices and one edge --- clearly \textsc{clique} is
still NP-complete for this slightly smaller class. But with this restriction,
we may assume that $k\geq2$ and therefore the right-hand side of \eqref{eq9}
is always defined.

Hence deciding if there exists an $h\in\mathbb{R}^{n+m}$ for which%
\[
\lbrack A_{G}(h,h,h)]^{2}>\frac{2}{27}\left(  1-\frac{1}{k-1}\right)
[h^{\top}h]^{3}%
\]
is an NP-hard problem. Note that $q=\frac{2}{27}[1-(k-1)^{-1}]\in\mathbb{Q}$
and so this problem is of the form:

\begin{problem}
\label{prob:clique}Given a symmetric $A\in\mathbb{Q}^{(n+m)\times
(n+m)\times(n+m)}$ and a positive $q\in\mathbb{Q}$, is it true that there
exists $h\in\mathbb{R}^{n+m}$ for which $[A(h,h,h)]^{2}>q[h^{\top}h]^{3}$?
\end{problem}

Let $\sigma\in\mathbb{Q}$, $\sigma>0$, be a self-concordance parameter and let%
\begin{equation}
\gamma:=\frac{1}{3}\left[  \frac{1}{2\sigma}\left(  1-\frac{1}{k-1}\right)
\right]  ^{1/3}.\label{eq:cube}%
\end{equation}
We follow the notation in Section~\ref{sec:sc}. Let $\Omega$ be the
$\varepsilon$-ball $B_{\varepsilon}(0)=\{x\in\mathbb{R}^{n}:\lVert
x\rVert<\varepsilon\}$ where $\varepsilon>0$ is to be chosen later. We are
interested in deciding self-concordance at $x=0$ of the cubic polynomial
$f:\Omega\rightarrow\mathbb{R}$ defined by%
\[
f(x)=\frac{\gamma}{2}x^{\top}x+A_{G}(x,x,x)=\frac{\gamma}{2}\sum
\nolimits_{i=1}^{n+m}x_{i}^{2}+\sum\nolimits_{i,j,k=1}^{n+m}a_{ijk}x_{i}%
x_{j}x_{k}.
\]
We have $\nabla^{2}f(0)=\gamma I$ where $I$ is the $(n+m)\times(n+m)$ identity
matrix. Since $\gamma>0$, i.e., $\nabla^{2}f(x)$ is strictly positive definite
in a neighborhood of $x=0$ and so there exists some $B_{\varepsilon}(0)$ on
which $f$ is convex --- this gives us our choice of $\varepsilon$. Also,
$\nabla^{3}f(0)=A_{G}$.

Hence $\nabla^{2}f(0)(h,h)=\gamma h^{\top}h=\gamma\lVert h\rVert_{2}^{2}$,
$\nabla^{3}f(0)(h,h,h)=A_{G}(h,h,h)$, and $f$ is self-concordant at the origin
with parameter $\sigma\in\mathbb{Q}$ if and only if%
\[
\lbrack A_{G}(h,h,h)]^{2}\leq4\sigma\gamma^{3}[h^{\top}h]^{3}=\frac{2}%
{27}\left(  1-\frac{1}{k-1}\right)  [h^{\top}h]^{3}%
\]
for all $h\in\mathbb{R}^{n+m}$. Note that this problem is of the form:

\begin{problem}
\label{prob:sc}Given a symmetric $A\in\mathbb{Q}^{(n+m)\times(n+m)\times
(n+m)}$ and a positive $q\in\mathbb{Q}$, is it true that for every
$h\in\mathbb{R}^{n+m}$, we have $[A(h,h,h)]^{2}\leq q[h^{\top}h]^{3}$?
\end{problem}

Mathematically, Problems~\ref{prob:clique} and \ref{prob:sc} are of course
equivalent, being logical complements of each other. However they may or may
not have the same computational complexity. By our discussion in
Section~\ref{sec:ct}, if our notion of reduction is the Cook reduction, then
we may indeed conclude that Problems~\ref{prob:clique} and \ref{prob:sc} are
equivalent in terms of computational complexity, i.e., deciding
self-concordance is NP-hard. However, if our notion of reduction is the Karp
reduction, then what we may deduce from the NP-hardness of
Problem~\ref{prob:clique} is that Problem~\ref{prob:sc} is co-NP-hard. In
either case, our conclusion is that self-concordance is intractable.

\begin{theorem}
\label{thm:main}Deciding whether a cubic polynomial is self-concordant at the
origin is NP-hard under Cook reduction and co-NP-hard under Karp reduction.
\end{theorem}

The argument in this section clearly works not just for cubic polynomials but
for any $f\in C^{3}(\Omega)$ as long as $0\in\Omega$, $\nabla^{2}f(0)=\gamma
I$, and $\nabla^{3}f(0)=A_{G}$ --- other derivatives and the remainder term in
the Taylor expansion of $f$ at $x=0$ may be chosen arbitrarily as long as $f$
stays convex in $\Omega$. This liberty allows one to extend the construction
above to functions with other desired properties. For instance, we may want an
example where $\Omega=\mathbb{R}^{n}$ and since cubic polynomials cannot be
convex on the whole of $\mathbb{R}^{n}$, we will need a quartic $f$ and
therefore need to choose $\nabla^{4}f(0)$ accordingly; or we may want an
example where $f$ is a barrier function, which is equivalent to $f$ having an
epigraph $\{(x,t)\in\mathbb{R}^{n+1}:x\in\Omega,\;f(x)\leq t\}$ that is
closed. One may trivially replace $0$ by any point $a\in\mathbb{R}^{n}$ by
considering the function $f_{a}(x)=f(x-a)$ on $\Omega=B_{\varepsilon}(a)$.

While we have proved our hardness result for functions on $\Omega
\subseteq\mathbb{R}^{n}$, it is easy to extend this to any $\mathbb{R}$-vector
space, for example, symmetric matrices $\mathbb{S}^{n\times n}$ or polynomials
$\mathbb{R}[x_{1},\dots,x_{n}]$, or even Riemannian manifolds with a
non-trivial class of geodesically convex functions (i.e., not just the
constant functions). Since self-concordance at a point is a local property, a
choice of coordinate patch would transform the problem to one over
$\mathbb{R}^{n}$; and by our remark at the end of Section~\ref{sec:sc}, it
will in fact be independent of our choice of coordinates.

Deciding self-concordance on the whole of $\Omega$ is of course at least as
hard as deciding self-concordance at a point in $\Omega$ and hence we have the following.

\begin{corollary}
For any $\Omega$ and any $\sigma>0$, deciding membership in $S_{\sigma}%
(\Omega)$ is NP-hard.
\end{corollary}

One may wonder why our conclusion in Theorem~\ref{thm:main} is stated as NP-
and co-NP-hardness as opposed to NP- and co-NP-completeness. It may appear
that given a \textsc{no} certificate $h\in\mathbb{R}^{n}$, it would be easy
(i.e., requires polynomial time) to decide whether $[A(h,h,h)]^{2}\leq
q[h^{\top}h]^{3}$ is indeed violated. But observe that it is only easy to
compute the quantities $[A(h,h,h)]^{2}$, $q[h^{\top}h]^{3}$, and compare their
magnitudes when we measure time complexity in units of \textit{real}
operations (i.e., arithmetic and ordering in $\mathbb{R}$). Since we measure
time complexity in units of \textit{bit} operations, even if the certificate
$h$ is in $\mathbb{Q}^{n}$, it could well have an exponential number of bits
and thus it is not at all clear that we may check $[A(h,h,h)]^{2}\leq
q[h^{\top}h]^{3}$ easily.

\section{Inapproximability of optimal self-concordance parameter}

Let $A\in\mathbb{Q}^{n\times n\times n}$ be symmetric and $f:\Omega
\rightarrow\mathbb{R}$ be defined by the cubic polynomial $f(x)=\frac{1}%
{2}x^{\top}x+A(x,x,x)$. As in Section~\ref{sec:NP}, $\Omega$ is chosen to be a
neighborhood of the origin so that $f$ is convex on $\Omega$. The condition
\eqref{eq:concord} for self-concordance of $f$ at $x=0$ with parameter
$\sigma>0$ may be written as
\begin{equation}
\lvert A(h,h,h)\rvert\leq2\sqrt{\sigma}\lVert h\rVert_{2}^{3} \label{eq1}%
\end{equation}
for all $h\in\mathbb{R}^{n}$. This is equivalent to requiring%
\begin{equation}
\max_{h\neq0}\frac{A(h,h,h)}{\lVert h\rVert_{2}^{3}}\leq2\sqrt{\sigma},
\label{eq2}%
\end{equation}
as $A(-h,-h,-h)=-A(h,h,h)$ and we may drop the absolute value in \eqref{eq1}.

Since $A\in\mathbb{R}^{n\times n\times n}$ is a symmetric $3$-hypermatrix, the
\textit{spectral norm} \cite{HL,L} of $A$,%
\[
\lVert A\rVert_{2,2,2}:=\max_{h_{1},h_{2},h_{3}\neq0}\frac{A(h_{1},h_{2}%
,h_{3})}{\lVert h_{1}\rVert_{2}\lVert h_{2}\rVert_{2}\lVert h_{3}\rVert_{2}%
}=\max_{h\neq0}\frac{A(h,h,h)}{\lVert h\rVert_{2}^{3}}.
\]
For the interested reader, the second equality above follows from Banach's
result on the polarization constant of Hilbert spaces \cite{B,PST}. Hence the
optimal self-concordance parameter of $f$ at $x=0$, i.e., the smallest value
of $\sigma$ so that \eqref{eq2} holds, is given by%
\begin{equation}
\sigma_{\operatorname*{opt}}=\frac{1}{4}\lVert A\rVert_{2,2,2}^{2}.\label{eq5}%
\end{equation}

The spectral norm of a $3$-hypermatrix is NP-hard to approximate to within a
certain constant factor by \cite[Theorem~1.11]{HL}, which we state here for
easy reference.

\begin{theorem}
[Hillar--Lim]\label{thm:HL} Let $A\in\mathbb{Q}^{n\times n\times n}$ and $N$
be the input size of $A$ in bits. Then it is NP-hard to approximate $\lVert
A\rVert_{2,2,2}$ to within a factor of $1-\varepsilon$ where
\[
\varepsilon=1-\left(  1+\frac{1}{N(N-1)}\right)  ^{-1/2}=\frac{1}%
{2N(N-1)}+O\left(  \frac{1}{N^{4}}\right)  .
\]

\end{theorem}

By \eqref{eq5} and Theorem~\ref{thm:HL}, $\sigma_{\operatorname*{opt}}$ is
NP-hard to approximate to within a factor of $\frac{1}{4}(1-\varepsilon)^{2}$
and consequently we have the following inapproximability result.

\begin{corollary}
There is no polynomial time approximation scheme for determining the optimal
self-concordance parameter $\sigma_{\operatorname*{opt}}$ unless
$\mathit{P}=\mathit{NP}$.
\end{corollary}

We refer the reader to \cite{HLZ, HS} for more extensive approximability
results and approximation algorithms (that are not \textsc{ptas}). In
particular, the results in \cite{HS} for quartic polynomials would apply to
the optimal second-order self-concordance parameter (see the next section).

\section{Complexity of deciding second-order self-concordance\label{sec:2sc}}

There is also an interesting notion of second-order self-concordance due to
Jarre \cite{J}. This requires that $f\in C^{4}(\Omega)$ and is given by a
condition involving the matrix $\nabla^{2}f(x)\in\mathbb{R}^{n\times n}$ and
the $4$-hypermatrix $\nabla^{4}f(x)\in\mathbb{R}^{n\times n\times n\times n}$.

\begin{definition}
[Jarre]If $\Omega\subseteq\mathbb{R}^{n}$ is a convex open set, then
$f:\Omega\rightarrow\mathbb{R}$ is said to be \textbf{self-concordant of order
two} with parameter $\tau>0$ at $x\in\Omega$ if%
\begin{equation}
\nabla^{2}f(x)(h,h)\geq0 \label{eq:2nd2}%
\end{equation}
and%
\begin{equation}
\nabla^{4}f(x)(h,h,h,h)\leq6\tau\left[  \nabla^{2}f(x)(h,h)\right]  ^{2}
\label{eq:concord2}%
\end{equation}
for all $h\in\mathbb{R}^{n}$; $f$ is self-concordant of order two on $\Omega$
if \eqref{eq:2nd2} and \eqref{eq:concord2} hold for all $x\in\Omega$.
\end{definition}

Note that%
\[
\nabla^{4}f(x)(h,h,h,h)=\sum\nolimits_{i,j,k,l=1}^{n}\frac{\partial^{4}%
f(x)}{\partial x_{i}\partial x_{j}\partial x_{k}\partial x_{l}}h_{i}h_{j}%
h_{k}h_{l},
\]
is a quartic polynomial in $h$ for any fixed $x\in\Omega$.

We follow the same argument in Section~\ref{sec:NP} to show that deciding
\eqref{eq:concord2} is NP-hard. This time the result would be deduced from
Motzkin--Strass Theorem except that for better parallelism with
Section~\ref{sec:NP}, we will use the quartic-maximization-over-sphere form
\eqref{eq4i} instead of the quadratic-maximization-over-simplex form \eqref{eq4b}.

Given a graph $G=(V,E)$ with $n$ vertices and $m$ edges where $n\geq2$ and
$m\geq1$, we define $A_{G}\in\mathbb{R}^{n\times n\times n\times n}$ by%
\[
a_{ijkl}=%
\begin{cases}
1 & i=k\text{, }j=l\text{, and }\{i,j\}\in E,\\
0 & \text{otherwise.}%
\end{cases}
\]
So $A=[a_{ijkl}]_{i,j,k,l=1}^{n}\in\mathbb{Q}^{n\times n\times n\times n}$ is
a symmetric $4$-hypermatrix. Now observe that, as in \eqref{eq4i},%
\[
\max_{\lVert h\rVert=1}A_{G}(h,h,h,h)=\max_{h\in\mathbb{S}^{n-1}}%
\sum\nolimits_{\{i,j\}\in E}h_{i}^{2}h_{j}^{2}=\frac{1}{2}\left(  1-\frac
{1}{\omega(G)}\right)
\]
by Motzkin--Strass Theorem. As in Section~\ref{sec:NP}, given an integer
$k\geq2$, deciding if a $k$-clique exists in $G$ is equivalent to deciding if
$\omega(G)>k-1$. Hence deciding if there exists $h\in\mathbb{R}^{n}$ with%
\begin{equation}
A_{G}(h,h,h,h)>\frac{1}{2}\left(  1-\frac{1}{k-1}\right)  [h^{\top}%
h]^{2}\label{eq10}%
\end{equation}
is NP-hard.

Given a second-order self-concordance parameter $\tau\in\mathbb{Q}$, $\tau>0$,
let%
\begin{equation}
\gamma:=\left[  \frac{1}{12\tau}\left(  1-\frac{1}{k-1}\right)  \right]
^{1/2}.\label{eq:quad}%
\end{equation}
We may now define $f:\Omega\rightarrow\mathbb{R}$ accordingly as the quartic
polynomial%
\[
f(x)=\frac{\gamma}{2}x^{\top}x+A_{G}(x,x,x,x)=\frac{\gamma}{2}\sum
\nolimits_{i=1}^{n}x_{i}^{2}+\sum\nolimits_{i,j,k,l=1}^{n}a_{ijkl}x_{i}%
x_{j}x_{k}x_{l}.
\]
Hence $\nabla^{2}f(0)(h,h)=\gamma h^{\top}h=\gamma\lVert h\rVert_{2}^{2}$ and
$\nabla^{4}f(0)(h,h,h,h)=A_{G}(h,h,h,h)$. Again we choose $\Omega$ to be a
neighborhood of the origin so that $f$ is convex on $\Omega$ as we did in
Section~\ref{sec:NP}. So the function $f$ is second-order self-concordant at
$x=0$ with parameter $\tau$ if and only if%
\begin{equation}
\lbrack A_{G}(h,h,h,h)\leq6\tau\gamma^{2}[h^{\top}h]^{2}=\frac{1}{2}\left(
1-\frac{1}{k-1}\right)  [h^{\top}h]^{2}\label{eq11}%
\end{equation}
is satisfied for all $h\in\mathbb{R}^{n}$. As in Section~\ref{sec:NP}, we
observe that the problem of deciding if there exists an $h\in\mathbb{R}^{n}$
satisfying \eqref{eq10} and the problem of deciding if \eqref{eq11} is
satisfied for all $h\in\mathbb{R}^{n}$ are logical complements. Since the
former is NP-hard, we arrive at the following conclusion:

\begin{theorem}
\label{thm:2sc}Deciding if a quartic polynomial is second-order
self-concordant at the origin is NP-hard under Cook reduction and co-NP-hard
under Karp reduction.
\end{theorem}

It has recently been shown that deciding various seemingly innocuous
properties of quartic polynomials \cite{AOPT, JLZ} all fall into the NP-hard
category, Theorem~\ref{thm:2sc} provides yet another such example.

\section{Conclusion}

As we have mentioned in the introduction, the hardness results here are
intended to shed light on the properties of self-concordance. They do not in
anyway invalidate the usefulness of the notion in practice since there are
basic principles that one may use to construct self-concordant (and
second-order self-concordant) functions for use as barriers in cone
programming --- see ``How to construct self-concordant barriers'' in
\cite[Chapter~5]{NN} for an extensive discussion or ``Self-concordant
calculus'' in \cite[Section~9.6]{BV} for a summary.

Self-concordance and second-order self-concordance are conditions involving
high-order tensors (orders $3$ and $4$ respectively), which is a topic of
great interest to the author. In particular, their NP-hardness serves as yet
reminder of the complexity of tensor problems \cite{HL}.

\section{Acknowledgement}

To be included.

\end{document}